\documentclass[12pt]{article}%
   
\usepackage{amsmath,enumerate, comment}
\usepackage{amsfonts}
\usepackage{amssymb, color, mathtools}

\newtheorem{theorem}{Theorem}[section]

\newtheorem{definition}[theorem]{Definition}

\newtheorem{example}[theorem]{Example}

\newtheorem{lemma}[theorem]{Lemma}

\newtheorem{proposition}[theorem]{Proposition}

\newtheorem{question}[theorem]{Question}
\newtheorem{remark}[theorem]{Remark}

\newenvironment{proof}[1][Proof]{\noindent\textbf{#1.} }
{\hfill \ \rule{0.5em}{0.5em}}

\newcommand{\coeff}[2]{( \lambda_{i_#1} - \lambda_{i_#2} ) } 
\newcommand{\coeffnp}[2]{ \lambda_{i_#1} - \lambda_{i_#2}  } 

\def\mike#1{\noindent
\textcolor{green}
{\textsc{(Mike:}
\textsf{#1})}}

\title{Hypergraphs of girth 5 and 6 and coding theory}

\author{Kathryn Haymaker\thanks{Department of Mathematics \& Statistics, Villanova University. \texttt{kathryn.haymaker@villanova.edu}} \and Michael Tait\thanks{Department of Mathematics \& Statistics, Villanova University. Research partially supported by NSF grant DMS-2245556. \texttt{michael.tait@villanova.edu}} \and Craig Timmons\thanks{Department of Mathematics and Statistics, California State University, Sacramento. \texttt{craig.timmons@csus.edu}}.
}
\begin{document}

\maketitle

\begin{abstract}
    In this paper, we study the maximum number of edges in an $N$-vertex $r$-uniform hypergraph with girth $g$ where $g \in \{5,6 \}$.  Writing 
$\textrm{ex}_r ( N, \mathcal{C}_{<g} )$ for this maximum, 
it is shown that $\textrm{ex}_r ( N , \mathcal{C}_{ < 5} ) = \Omega_r ( N^{3/2 - o(1)} )$ for $r \in \{4,5,6 \}$.  We address an unproved claim from \cite{TV} asserting a technique of Ruzsa can be used to show that this 
lower bound holds for all $r \geq 3$.  We carefully explain one of the main obstacles that was overlooked at the time the claim from \cite{TV} was made, and 
show that this obstacle can be overcome when $r\in \{4,5,6\}$.  We use constructions from coding theory to prove 
nontrivial lower bounds that hold for all $r \geq 3$. Finally, we use a recent result of Conlon, Fox, Sudakov, and Zhao to show that the sphere packing bound from coding theory may be improved when upper bounding the size of linear $q$-ary codes of distance $6$.
\end{abstract}

\section{Introduction}

Let $X$ be a finite set and $\binom{X}{r}$ be the collection of all subsets of $X$ with $r$ elements.
An \emph{$r$-uniform hypergraph} $\mathcal{H}$ with vertex set $X$ is a subset of $\binom{X}{r}$.  
Being rather general objects, many combinatorial problems can be phrased in terms of uniform hypergraphs.
Furthermore, there are important instances in which this perspective is useful.  Indeed, hypergraphs can provide a framework for addressing problems that at first seem unrelated, yet are linked by some underlying idea or concept.      

As early as 2000, researchers have used results and methods from hypergraph Tur\'{a}n theory to 
prove bounds on the sizes of codes.  An important instance of this is the work of Alon, Fischer, and Szegedy \cite{alonFischerSzegedy} from 2001.  They used tools from additive combinatorics and extremal hypergraph theory to study IPP codes.  
Let $V$ be a set with $n$ elements and $\mathcal{C} \subseteq V^4$.  We say that $\mathcal{C}$ has the \emph{identifiable parent property} (IPP) if  
\begin{enumerate}
    \item for all distinct $a,b,c \in \mathcal{C}$, there is a coordinate $i \in \{1,2,3,4 \}$ where $a_i$, $b_i$, and $c_i$ are all different, and 
    \item for all $a,b,c,d \in \mathcal{C}$ with $\{a,b \} \cap \{c , d \} = \emptyset$, there is a coordinate $j \in \{1,2,3,4 \}$ such that 
    $\{a_j , b_j \} \cap \{ c_j , d_j \} = \emptyset$.  
\end{enumerate}
Let $f(n)$ be the the maximum size of an IPP code $\mathcal{C} \subseteq V^4$.
Given $\epsilon > 0$, Alon, Fischer, and Szegedy proved that for all $n > n_0 ( \epsilon )$, 
\begin{equation}\label{eq:alon IPP bound}
    n^{2- \epsilon } < f(n) < \epsilon n^2 .
    \end{equation}
Their proof of the upper bound takes a code with IPP and defines a corresponding hypergraph.  
The graph removal lemma \cite{alonDuke, szemeredi} is then used to prove an upper bound on the number of edges in this hypergraph.  The proof of the lower bound defines a hypergraph 
based on an extension of Behrend's construction of sets with no 3-term arithmetic progression \cite{Ruzsa}.   
Alon \cite{alon} used a similar partite hypergraph construction to prove a theorem on property testing in graphs.  Again, one of the ingredients in the construction is Ruzsa's generalization of Behrend's construction.  Something notable here is that the application is on
property testing of graphs, an area which at first glance may have no obvious connection to coding theory.  

Still in the early 2000s, Lazebnik and Verstra\"{e}te \cite{LV} made a significant contribution to hypergraph Tur\'{a}n theory.  They determined an asymptotic formula for the maximum number of edges in an $N$-vertex 3-uniform hypergraph with girth five, defined in the Berge sense.  One of the elements in their proof is a hypergraph construction similar to the one used 
in \cite{alon, alonFischerSzegedy}, 
but now the generalized Behrend construction is replaced with a special type of Sidon set.   

Over the past 20 years, using tools from additive number theory to construct codes and hypergraphs has 
evolved in sophistication.  Many papers have used this method in various forms, such as a taking a sparse hypergraph with many edges and viewing it as a code inside some Hamming space.  One then uses properties of the hypergraph to deduce properties of the code.    

In 2020, several of these connections were made explicit by Shangguan and Tamo \cite{ShangguanTamo1} who constructed sparse hypergraphs to make progress on problems involving three different types of codes: (i) parent-identifying systems,
(ii) combinatorial batch codes, and (iii) locally recoverable codes.  
One condition that is often used to make a graph or hypergraph sparse is to forbid short cycles.  
Recall a {\em Berge cycle} of length $k$ in a hypergraph is a sequence of $k$ distinct vertices $v_1,\cdots, v_k$ and $k$ distinct edges $E_1,\cdots, E_k$ such that $\{v_i, v_{i+1}\} \subset E_i$, where indices are read modulo $k$.  A hypergraph $\mathcal{H}$ has girth at least $g$ if $\mathcal{H}$ has no Berge cycles of length $k$ for every $2 \leq k \leq g-1$. We will use the notation $\mathcal{C}_{<g}$ to denote the family of Berge cycles of length at most $g-1$. Hence, a hypergraph has girth at least $g$ if and only if it is $\mathcal{C}_{< g}$-free. Given a family of $r$-uniform hypergraphs $\mathcal{F}$, the {\em Tur\'an number} for $\mathcal{F}$ is the maximum number of edges in an $r$-uniform $N$-vertex hypergraph which is $\mathcal{F}$-free and is denoted by $\mathrm{ex}_r(N, \mathcal{F})$. 
Because we will use objects from both coding theory and graph theory in this paper, $n$ will denote the dimension of a vector space that a code lives in and $N$ will denote the number of vertices in a graph. An $[n,k,d]_q$ code is a linear subspace of $\mathbb{F}_q^n$ of dimension $k$ such that every nonzero vector has at least $d$ nonzero entries. 

With one motivation coming from coding theory, 
Shangguan and Tamo used a result on hypergraph 
independence numbers \cite{DLR} and the probabilistic method to construct sparse hypergraphs.  They noted that some of these hypergraphs imply lower bounds on the Tur\'{a}n number of $\mathcal{C}_{ < g}$-free graphs.  
Indeed, one of the corollaries to the main theorem in \cite{ShangguanTamo1} is that for all $r \geq 3$ and $g \geq 5$, 
$\mathrm{ex}_r ( N , \mathcal{C}_{ < g} )  = \Omega_r ( N^{\frac{g-1}{g-2} } ( \log N )^{ \frac{1}{ g - 2} })$. 
This gave a logarithmic factor improvement to the bound 
obtained by Xing and Yuan who also connected $\mathcal{C}_{ <g }$-free $r$-uniform hypergraphs with locally recoverable codes \cite{XingYuan}.
They proved that a construction giving a lower bound 
on $\textrm{ex}_r( N , \mathcal{C}_{ < 5} )$ implies a lower bound on the length of a locally recoverable code (LRC) with minimum distance $d$ and locality $r$, where $d \in \{ 9 , 10 \}$ and $r \geq d - 2$.  There is a large amount of research on LRCs, which adds to the motivation for continued study of the $\mathcal{C}_{ < g}$-free $r$-uniform hypergraphs.  As mentioned earlier, the case $r = 3$, $g = 5$ was asymptotically solved by 
Lazebnik and Verstra\"{e}te \cite{LV}, a paper which also 
posed an important theoretical question on generalized Sidon sets.  The case of $r = 3$ and odd $g \geq 5$ was shown to be related to multiplicative square-free sequences by Pach and Vizer \cite{PachVizer}. Specifically, for $k \geq 2$, let $F_k (N)$ be the maximum size of a set $A \subset \{1,2, \dots , N \}$ such that no product of $k$ distinct integers from $A$ is a square.  Theorem 11 \cite{PachVizer}  
shows that a lower bound on $\textrm{ex}_{3}( N , \mathcal{C}_{<5} )$ implies a lower bound on 
$F_{8}(N)$. Gerbner and Patk\'{o}s \cite{GerbnerPatkos} proved that $N$-vertex $(a+b)$-uniform $\mathcal{C}_{ < 5}$-free hypergraphs can be used to prove lower bounds on generalized Tur\'{a}n numbers involving $K_{a,b}$.
Let $\textrm{ex}(N, H,F)$ be the maximum number of copies of $H$ in an $N$-vertex $F$-free graph.
Proposition 4.1 of \cite{GerbnerPatkos} shows that $\textrm{ex}(N, K_{a,b} , K_{2,t}) 
\geq \textrm{ex}_{a+b}(N, \mathcal{C}_{< 5} )$ 
for $2 \leq a \leq b < t$.  Therefore, lower bounds on the girth 5 Tur\'{a}n problem have direct applications to problems in extremal combinatorics.  
Circling back to coding theory, Xing and Yuan \cite{XingYuan2} proved an 
equivalence between certain types of $\mathcal{C}_{ <g}$-free uniform hypergraphs and LRCs.  Their equivalence reflects the fact that the more edges in the hypergraph often results in a code with better parameters. 

Given that several published results use the best available lower   
bounds on Tur\'{a}n numbers of $r$-uniform $\mathcal{C}_{ < 5}$-free hypergraphs, we now focus the discussion on the best proven bounds that we are aware of.
As noted above, Lazebnik and Verstra\"ete gave the asymptotic formula $\mathrm{ex}_r(N, \mathcal{C}_{<5}) \sim \frac{1}{6}N^{3/2}$. For $r>3$ the situation is less clear. 
The third author and Verstra\"{e}te \cite{TV} claimed that 
one could adapt a method of Ruzsa \cite{Ruzsa} to show that $\mathrm{ex}_r(N, \mathcal{C}_{<5}) = N^{3/2-o(1)}$ for all $r$.   However, despite that almost 10 years have passed since \cite{TV} appeared, we currently do not have a proof.  Several papers have repeated this claim \cite{TimmonsThesis, Spiro3, Spiro1, Spiro2} or make use of it \cite{GerbnerPatkos, PachVizer, ShangguanTamo1, XingYuan}.  
One of the contributions of this paper is to highlight and identify an obstacle that 
was not foreseen, and then prove that we may use Ruzsa's method and overcome it when $r \in \{4,5,6 \}$.  
We do not know how to overcome this obstacle for larger $r$.  We describe at a high level what the significant gap is at the beginning of Section \ref{sec: ruzsa method} and exactly which detail of the claimed proof does not go through in Remark \ref{remark: subtlety}.    
For all $r>3$, 
the best-known lower bound is from the logarithmic improvement to the probabilistic method \cite{ShangguanTamo1}, which gives $\mathrm{ex}_r(N, \mathcal{C}_{<5}) = \Omega \left(N^{4/3}(\log N)^{1/3}\right)$ and $\mathrm{ex}_r(N, \mathcal{C}_{<6}) = \Omega \left(N^{5/4} (\log N)^{1/4}\right)$. The upper bound for both functions is $O_r(N^{3/2})$ \cite{GL}. 

For larger $r \geq 7$, we provide a different construction which utilizes objects from coding theory and improves the current lower bounds on $\mathrm{ex}_r(N, \mathcal{C}_{<5})$ and $\mathrm{ex}_r(N, \mathcal{C}_{<6})$.



\subsection{Our results}

In the papers mentioned above, number theoretic properties are used to produce hypergraphs that obey certain properties. In this paper we similarly construct hypergraphs where subgraphs can be described by solutions to given equations, which we then avoid to obtain a hypergraph with a given girth. The papers described above either use this type of construction to study Tur\'an problems directly or use this type of construction to apply graph theoretic methods on coding theory problems. Our main results are on Tur\'an problems, but go the other direction from the early papers. That is, we use results from coding theory to give improvements on Tur\'an problems rather than the other way. At the end of the paper we go back in the more well-studied direction and show that recent hypergraph Tur\'an results imply an improvement on the sphere-packing bound for linear codes of distance $6$. Our contributions are as follows:

\begin{itemize}
\item We address the unproved claim 
of the bound lower bound $N^{3/2-o(1)}$ for $\mathcal{C}_{<5}$-free graphs of uniformity larger than 3.  We explain why the claim was erroneously made by going through Ruzsa's method carefully, and noting that such a claim relies on a solution to a finite field problem that is to our knowledge open. This problem is quite interesting in its own right (see Question \ref{behrend over field question} in the Conclusion).
\item We show that when $r \in \{4,5,6 \}$, this obstacle can be overcome in a different way than answering Question \ref{behrend over field question} affirmatively (Theorem \ref{thm: 456}).    
\item For all other $r$, we use constructions from coding theory to give lower bounds on $\mathrm{ex}_r(N, \mathcal{C}_{<5})$ and $\mathrm{ex}_r(N, \mathcal{C}_{<6})$ that improve the probabilistic bound (Theorem \ref{thm: graphs from codes}).
\item We show that recent results of Conlon, Fox, Sudakov, and Zhao \cite{conlon} improve the sphere packing bound on the specific case of linear codes of distance $6$ for any odd prime power $q\geq 7$ (Theorem \ref{distance 6 code upper bound}). 
\end{itemize}

\begin{theorem}\label{thm: 456}
    Let $r\in \{4,5,6\}$. Then we have 
    \[
    \mathrm{ex}_r(N, \mathcal{C}_{<5})  = \Omega_r\left(N^{3/2-o(1)}\right).
    \]
\end{theorem}

\begin{theorem}\label{thm: graphs from codes}
    Let $r$ be fixed. Then there exists constants $c_r$ and $c_r'$ such that \begin{enumerate}[(a)]\item $\mathrm{ex}_r(N, \mathcal{C}_{<5})  \geq c_r \cdot N^{10/7}$, and
 \item $\mathrm{ex}_r(N, \mathcal{C}_{<6}) \geq c_r'\cdot N^{4/3}$.
    \end{enumerate}
\end{theorem}


The constructions used to prove Theorems \ref{thm: graphs from codes} and \ref{thm: 456} can be used to give upper bounds on the size of linear codes. In particular, for linear $[n,k,6]_q$ codes, as $n$ goes to infinity, the sphere-packing bound and Johnson bound \cite{johnson1962new, roth2006} both give
\begin{equation}\label{sphere packing}
k\leq n - 2\log_q n - O(1).
\end{equation}
When the distance of a code grows with $n$ there are improvements to these bounds, but in the regime where the distance is fixed and $n$ is going to infinity, \eqref{sphere packing} is the best published result that we are aware of. We improve \eqref{sphere packing} by an additive factor going to infinity.

\begin{theorem}\label{distance 6 code upper bound}
  Let $q \geq 7$ be a power of an odd prime. If $\mathcal{C}$ is an $[n,k,6]_q$ code, then 
    \[
    k\leq n - 2\log_q n - \omega(1),
    \]
    where $\omega(1)$ is a function that goes to infinity with $n$.
\end{theorem}

We note that Theorem \ref{distance 6 code upper bound} requires the code to be linear and does not give any bound on the size of general distance $6$ codes.

This paper is organized as follows. In Section \ref{sec: setup} we set up the method and explore some of the subtleties involved in linking Berge cycles with equations. In Sections \ref{sec: codes}, \ref{sec: ruzsa method}, and \ref{sec: sphere packing} we prove Theorems \ref{thm: graphs from codes}, \ref{thm: 456}, and \ref{distance 6 code upper bound} respectively. We end the paper with some concluding remarks.



\section{Ruzsa-Szemer\'{e}di Hypergraphs}\label{sec: setup}

In their solution to the $(6,3)$-problem, Ruzsa and Szemer\'{e}di 
constructed 3-partite graphs with the property that each  
edge is in exactly one triangle \cite{RS}.  
There are different ways to present their construction, but one of essential steps is to choose a set $A \subset \mathbb{Z}_p$ with no 3-term arithmetic progression.  Assuming such a set has been chosen, let $G_p(A)$ be the graph whose vertex set is three disjoint copies of $\mathbb{Z}_p$.  
The edges of $G_p(A)$ are obtained by taking 
the union of all triangles of the form 
\[
(x,x+a,x+2a) ~\mbox{where}~ x \in \mathbb{Z}_p, a \in A.
\]
Here $x$ is the vertex in the first copy of $\mathbb{Z}_p$, $x + a$ is the vertex in the second copy, and $x + 2a$ is in the third copy of $\mathbb{Z}_p$.  
Ruzsa and Szemer\'{e}di proved that the only triangles in 
$G_p(A)$ are the $p |A|$ triangles used to define
the edge set.  To maximize the number of edges, which is $3p|A|$, the set $A$ is chosen to be as dense as possible, which is where the classical construction of Behrend enters the picture.  

This approach can also be stated in hypergraph language using Berge cycles.  Let 
$\mathcal{H}_p( A)$ be the 3-uniform 3-partite hypergraph with vertex set $\cup_{i=1}^{3} ( \mathbb{Z}_p \times \{ i \} )$ and edge set  
\[
\{  \left(   (x,1) , (x + a , 2) , (x + 2a , 3) \right)  : x \in \mathbb{Z}_p , a \in A \}.
\]
Proving that $\mathcal{H}_p( A )$ has girth 4
is equivalent to proving that every edge of $G_p(A)$ is in exactly one triangle.    

Subsequent research has implemented extensions of Ruzsa-Szemer\'{e}di hypergraphs where more parts are considered \cite{alon, alonFischerSzegedy, Ge, LV}.  Also, sets $A$ which avoid solutions to other equations, like the Sidon equation $X_1 + X_2 - X_3 - X_4 = 0$, are often used depending on the Berge cycles one wants to avoid.      
Let us now define these hypergraphs in a general context.  

\begin{definition} 
Let $r \geq 2$ be an integer and $R$ and $S$ be rings with unity where $R$ is an $S$-module.    
Let $\vec{ \lambda} = ( \lambda_1 , \lambda_2 , \dots , \lambda_r)$ be a vector whose entries are elements of $S$. 
For finite subsets $R' \subset R$ and $A \subset R$, 
define $\mathcal{H}( A , \vec{ \lambda } )$ to be the $r$-uniform $r$-partite hypergraph with vertex set 
$\cup_{i=1}^r ( R' \times \{ i \} )$ 
and edge set 
\[
\bigcup_{x \in R' ,a \in A } \{ \left( ( x + \lambda_1 a ,1) 
, ( x+ \lambda_2 a ,2 ) ,  \dots , 
 (x  + \lambda_r a , r ) \right) \}.  
\]
\end{definition}

Write $e(x,a)$ for the edge $\{ ( ( x + \lambda_1 a ,1) , ( x + \lambda_2 a , 2  ) , \dots , ( x + \lambda_r a ,r ) ) \}$.  We view the vertex $(x + \lambda_i a , i)$ as contained in the $i$-th copy of $R'$, which is $R' \times \{i \}$.  Some authors choose to omit the second coordinate.  We have opted to include the coordinate because it avoids using the same notation for vertices in different parts.   
Choices that have been used for $R$ and $S$ are  
(i) $R = S = \mathbb{Z}$ \cite{alon, alonFischerSzegedy, Ge} and 
(ii) $R = \mathbb{Z}_n$, $S = \mathbb{Z}$ \cite{Ge-Hash, LV}.   
In our applications, we will take $R= \mathbb{F}_q^t$ and $S = \mathbb{F}_q$ in Section \ref{sec: codes}, 
and $S = \mathbb{Z}_p$, $R = \mathbb{Z}_p^2$ in Section \ref{sec: ruzsa method}.  

\subsection{Berge cycles and corresponding equations}

One of the key properties of $\mathcal{H}( A , \vec{ \lambda})$ is that a Berge cycle implies 
that there is an equation over $R$, whose coefficients are differences of the entries of $\vec{ \lambda}$, and 
the variables must be elements in $A$.  The following proposition is often used, although perhaps not stated in this general way.  A sketch of the proof is that if a vertex is in two edges, it implies an equation over $R$ is satisfied. Adding up all of these equations, rearranging and simplifying gives \eqref{a equation}.
A detailed proof can easily be obtained from existing literature (\cite{Ge} for instance).

\begin{proposition}\label{prop:cycle equations} 
If the hypergraph $\mathcal{H}( A , \vec{ \lambda} )$ contains a cycle of length $k$, then there are indices $i_1,i_2, \dots ,i_k$ such that 
$i_\ell \not = i_{\ell+1}$ for $1\leq \ell \leq k-1$, $i_k\not = i_1$, and 
\begin{equation}\label{a equation}
(\lambda_{i_2} - \lambda_{i_1})a_1 + (\lambda_{i_3}- \lambda_{i_2})a_2 + \cdots + (\lambda_{i_k} - \lambda_{i_{k-1}})a_{k-1} + (\lambda_{i_1} - \lambda_{i_k})a_k = 0
\end{equation}
for some $a_1,a_2, \dots , a_k \in A$.  
\end{proposition}

Without assumptions on $\vec{\lambda}$, it could be the case that (\ref{a equation}) is a trivial equation 
which holds for ``most" values that the variables can take.  
If $\vec{\lambda} = (1, 1, \dots , 1)$ where 1 is the unity element of $R$, 
then (\ref{a equation}) holds for all $a_1,a_2, \dots , a_k \in A$ because each coefficient in (\ref{a equation}) is 0.  
For this reason, all of the constructions that we are aware of choose 
the entries of $\vec{ \lambda}$ to be distinct to avoid an equation where a coefficient is 0.  

Assuming the entries of $\vec{ \lambda}$ are distinct, equation (\ref{a equation}) is a homogeneous equation in $k$ variables.  There is a fundamental difference between the cases when $k = 2$ and when $k \geq 3$.  Let us quickly dispense of the case $k = 2$ so that the focus can be on $k \geq 3$.  
In most applications, including ours, the ring $R$ has the cancellation property.  The cancellation property and distinct entries in $\vec{\lambda}$ are all that is needed to avoid 2-cycles.  

\begin{proposition}\label{prop: linear}
If for all $i_1 \neq i_2$ and $a_1 , a_2 \in A$, 
the equation 
\[
( \lambda_{i_2} - \lambda_{i_1} ) a_1 =
( \lambda_{i_2} - \lambda_{i_1} ) a_2
\]
implies that $a_1 = a_2$, then $\mathcal{H}( A , \vec{ \lambda} )$ has no Berge 2-cycle.   
\end{proposition}
\begin{proof}
Suppose that $e (x_1,a_1)$ and $ e (x_2 ,a_2)$ are two edges that form a 2-cycle.  
By Proposition \ref{prop:cycle equations} there are indices $i_1$ and $i_2$ with $i_1 \neq i_2$ and 
\[
( \lambda_{i_2} - \lambda_{i_1} ) a_1 + ( \lambda_{i_1} - \lambda_{i_2} ) a_2 = 0,  
\]
which can be rewritten as $( \lambda_{i_2} - \lambda_{i_1} ) a_1 = ( \lambda_{i_2} - \lambda_{i_1} ) a_2$.  The assumption then implies that $a_1 = a_2$.  Since there is a vertex $(w, i_1)$ in the 
$i_1$-th part that is in both $e(x_1,a_1)$ and $e(x_2,a_2)$, we have 
\[
w = x_1 + \lambda_{i_1} a_1 ~~\mbox{and} ~~ w  = x_2 + \lambda_{i_1}a_2.
\]
Combining these two equations with $a_1 = a_2$ gives $x_1 = x_2$, and so $e(x_1, a_1) = e(x_2 , a_2)$ which is a contradiction.  
We conclude that $\mathcal{H}(A, \vec{ \lambda} )$ does not have a 2-cycle.      
\end{proof}

\subsection{The genus of a Berge cycle and trivial solutions}

For $k \geq 3$, equation (\ref{a equation}) is a homogeneous, invariant equation.  
When the coefficients are integers, these equations were studied by Ruzsa \cite{Ruzsa}.
Let us recall some important terms introduced 
in \cite{Ruzsa}.  Consider the equation 
\begin{equation}\label{genus def equation}
\sum_{i=1}^{k} \beta_i X_i = 0
\end{equation}
where each $\beta_i$ is a non-zero element of $S$ and $\sum_{i=1}^k \beta_i = 0$. 
For any $a \in A$, setting $X_1 = X_2 = \dots = X_k  = a$ solves (\ref{genus def equation}).  These trivial solutions where all variables are equal to the same value will always occur.  
Depending on the equation, there may be trivial solutions (definition recalled in the next paragraph) using multiple elements from $A$.    

Suppose that $(X_1 , \dots , X_k ) = (a_1 , \dots , a_k)$ is a solution to (\ref{genus def equation}), and that this solution uses exactly $\ell$ distinct elements of $A$.  Let $[k] = T_1 \cup \dots \cup T_{ \ell}$ be a partition such that $a_i = a_j$ if and only if $i,j \in T_v$ for some $v$.
We call $(X_1, \dots , X_k) = (a_1 , \dots , a_k)$ a \emph{trivial solution} if $\sum_{i \in T_v}  \beta_i = 0$ for all $v \in \{1 ,2 , \dots , \ell \}$.  
We define the \emph{genus} of (\ref{genus def equation}) to be the largest $m$ such that there is partition $T_1 \cup T_2 \cup \dots \cup T_m$ of 
the set of subscripts $[ k ] $ into non-empty sets such that 
\[
\sum_{i \in T_v } \beta_i = 0 ~~~\mbox{for all $1 \leq v \leq m$.}
\]
The genus of an equation was originally defined by Ruzsa \cite{Ruzsa} in the case that $R=S = \mathbb{Z}$.  The same definition makes sense for other $R$ and $S$. For example, \cite{MT} defines the genus of an equation in the same way when $R = \mathbb{F}_q^t$ and $S=\mathbb{F}_q$. In the case that each $\beta_i$ is an integer, then we say that equation (\ref{genus def equation}) is of \emph{type $(k_1,k_2)$} if there are $k_1$ positive coefficients and $k_2$ negative coefficients.  

Since a Berge $k$-cycle in $\mathcal{H}(A , \vec{ \lambda})$ corresponds to an 
invariant, homogeneous equation in $k$ variables, we \emph{define the genus of the $k$-cycle} to be the genus of the corresponding equation (\ref{a equation}). 
More precisely, suppose $E_1, E_2, \dots ,E_k$ are the edges of a Berge $k$-cycle.  Since $\mathcal{H}(A , \vec{ \lambda}) $ has no 2-cycle, each pair of consecutive edges on the cycle intersects in exactly one vertex.  For $1 \leq j \leq k -1$, let $(w_{i_j} , i_j )$ be the unique vertex in $E_j \cap E_{j+1}$, and 
let $(w_{i_k} , i_k)$ be the unique vertex in $E_k \cap E_1$.  This gives a unique sequence of part indices $i_1 , i_2 , \dots , i_k$ such that the $j$-th vertex on the cycle is in part $i_j$.  Furthermore,
$i_1 \neq i_2$, $i_2 \neq i_3, \dots , i_k \neq i_1$ since no single edge in $\mathcal{H}( A , \vec{ \lambda})$ contains two vertices from the same part.  
Cyclically shifting the order of the edges $E_1, E_2 ,\dots , E_k$ will cyclically shift the parts, but other than the notation, it will not change equation (\ref{a equation}).  

Having defined the genus of Berge $k$-cycle in $\mathcal{H}(A , \vec{ \lambda})$, we now state 
a useful property that connects adjacent edges on a cycle and repeated $A$ values.  

\begin{proposition}\label{prop:consecutive edges}
    Let $x_1,x_2 \in R'$ and $a \in A$.  If $e(x_1 , a)$ and $e(x_2 , a)$ share at least one vertex, then $x_1 = x_2$, i.e., the two edges are the same. 
 \end{proposition}
\begin{proof}
    Suppose that there is a vertex a vertex $( w , i_1)$ in both $e(x_1, a)$ and $e(x_2 , a)$.  We then have $w = x_1 +  \lambda_{i_1} a = x_2 + \lambda_{i_1} a$ which implies $x_1 = x_2$.    
\end{proof}

\medskip
Propositions \ref{prop:cycle equations}, \ref{prop: linear}, \ref{prop:consecutive edges}, and 
the genus of a Berge cycle is enough to provide a high-level description of several existing applications of Ruzsa-Szemer\'{e}di hypergraphs.  First $R$, $R'$, $S$, and $\vec{ \lambda}$ are chosen.  This determines the algebraic setting and the equations that come from Berge cycles.    
The next step is to choose $A$ as large as possible subject to the condition that $A$ has only trivial solutions to the equations coming from cycles that one wants to avoid.  In \cite{alon, alonFischerSzegedy, Ge-Hash, Ge}, every avoided cycle corresponds to a genus 1 equation with bounds on the coefficients.  Although each must overcome its unique technical challenges, there exist constructions of large sets that avoid solutions to families of genus 1 equations of type $(\ell,1)$.  We introduce these sets now.   

Let $C , \ell \geq 2$ be integers.  Let $\mathcal{E}_{C , \ell}$ be the collection of all equations of the form 
\[
\beta_1 X_1 + \beta_2 X_2 + \dots + \beta_{k-1} X_{k-1} 
= \beta_k X_k
\]
where $k \in \{2,3, \dots , \ell \}$, 
$\beta_1, \beta_2 , \dots , \beta_{k-1} \in \{1,2, \dots , C \} \subset \mathbb{Z}$, and $\beta_1 + \beta_2 + \dots + \beta_{k-1} = \beta_k$.  

\begin{lemma}\label{lemma:multiple Behrend}
For every pair of positive integers $C$ and $\ell$, there are real numbers $N, \tau , \gamma >0$ such that the following holds.  For each prime $p > N$, there is a 
set $ B \subset \mathbb{Z}_p$ with $|B| > \gamma p e^{ - \tau \sqrt{ \log p } } $ and $B$ has only trivial solutions to each equation in $\mathcal{E}_{C, \ell}$.   \end{lemma}

Lemma \ref{lemma:multiple Behrend} follows from a standard averaging argument and a Behrend-type construction.  All of the ideas needed to prove it can be found in Lemmas 3.4 and 3.8 of \cite{Ge}.  
It is worth noting that Lemma 3.5 of \cite{Ge} concerns equations of types different from $( \ell , 1)$ but which are all still genus 1.

\subsection{Avoiding 3-cycles}

By Proposition \ref{prop:cycle equations}, a 3-cycle leads to an equation of the form 
\begin{equation}\label{eq:3 cycle equation}
( \lambda_{i_2} - \lambda_{i_1} ) a_1 + 
( \lambda_{i_3} - \lambda_{i_2} ) a_2 + 
( \lambda_{i_1} - \lambda_{i_3} ) a_3 =0
\end{equation}
with $i_1 \neq i_2$, $i_2 \neq i_3$, $i_3 \neq i_1$ and 
$a_1,a_2,a_3 \in A$.  These three non-equalities
imply that the vertices of the 3-cycle must be in three different parts (similarly to how a 2-cycle must also have the vertices of the cycle in different parts).  
Since we are assuming that the $\lambda_i$'s are distinct
elements of $S$, each coefficient in (\ref{eq:3 cycle equation}) is not zero.  Therefore, 
(\ref{eq:3 cycle equation}) is always a genus 1 equation and any trivial solution will have $a_1 = a_2 = a_3$. When all variables are equal, Proposition \ref{prop:consecutive edges} applies to any pair of consecutive edges on the cycle.  See Section \ref{sec:avoiding cycles by genus type} for further remarks on genus 1 $k$-cycles, of which 3-cycles are a special case.    


\subsection{Avoiding 4-cycles}

By Proposition \ref{prop:cycle equations} a 4-cycle leads to the equation 
\begin{equation}\label{eq:4 cycle equation}
( \lambda_{i_2} - \lambda_{i_1} ) a_1 + 
( \lambda_{i_3} - \lambda_{i_2} ) a_2 + 
( \lambda_{i_4} - \lambda_{i_3} ) a_3 +
( \lambda_{i_1} - \lambda_{i_4} ) a_4 
=0
\end{equation}
with $i_1 \neq i_2$, $i_2 \neq i_3$, $i_3 \neq i_4$, 
$i_4 \neq i_1$, and $a_1,a_2,a_3 ,a_4 \in A$.  The 
four non-inequalities tell us that the vertices of a 4-cycle could be in two, three or four parts.  

If the vertices are in two parts, then $i_1 = i_3$, $i_2 = i_4$, and (\ref{eq:4 cycle equation}) can be rewritten as 
\begin{equation}\label{eq:4 cycle Sidon eq}
( \lambda_{i_2} - \lambda_{i_1} ) a_1 + 
( \lambda_{i_2} - \lambda_{i_1} ) a_3 = 
( \lambda_{i_2} - \lambda_{i_1} ) a_2 +
( \lambda_{i_2} - \lambda_{i_1} ) a_4. 
\end{equation}
As $\lambda_{i_2} - \lambda_{i_1} \neq 0$, we may cancel to obtain $a_1 + a_3  = a_2 + a_4$.  
This shows that $A$ has a solution to the Sidon equation
$X_1  - X_2  + X_3 - X_4 = 0$.  This equation has genus 2 so that there are trivial solutions using any two distinct elements of $A$.  
Nevertheless, if $A$ only has trivial solutions to
$X_1 - X_2 + X_3 - X_4 = 0$, then $a_1 - a_2 + a_3 - a_4 = 0$ implies $\{a_1 , a_3 \} = \{ a_2 , a_4 \}$.  Thus,  $a_1 = a_2$ or $a_1 = a_4$ and in both cases, we will be able to apply Proposition \ref{prop:consecutive edges}.  

Although far more general than what is needed here, Ruzsa \cite{Ruzsa} proved that in the integer case, if $A \subset [N]$ has only trivial solutions to a genus $g$ equation with integer coefficients, then $|A|  = O (N^{1/g})$. 
This suggests that $|A|$ will decrease significantly if the cycles we want to avoid have genus $g > 1$.  

If the vertices of the 4-cycle are in three parts, then without loss of generality, we may assume 
$i_1 = i_3$, $i_2 \neq i_4$.  Here (\ref{eq:4 cycle equation}) becomes
\begin{equation}\label{eq:symmetric 2 2}
( \lambda_{i_2} - \lambda_{i_1} ) a_1 +
( \lambda_{i_4} - \lambda_{i_1} ) a_3 
=
( \lambda_{i_2} - \lambda_{i_1} ) a_2 +
( \lambda_{i_4} - \lambda_{i_1} ) a_4.
\end{equation}
This is a genus 2 equation.  If $A$ has only trivial solutions to (\ref{eq:symmetric 2 2}), then 
$\{a_1 , a_3 \} = \{ a_2  ,a_4 \}$.  
Hence, $a_1 = a_2$ or $a_1 = a_4$ and in both cases 
Proposition \ref{prop:consecutive edges} applies.  

Lastly, if the vertices of the 4-cycle are in four parts,
then all of $i_1,i_2,i_3,i_4$ are distinct.  Berge $k$-cycles with vertices in distinct parts in $\mathcal{H}( A , \vec{ \lambda} )$ are called \emph{rainbow} in \cite{Ge}.  Equation (\ref{eq:4 cycle equation}) is now   
\begin{equation}\label{eq:variables 4 cycle equation}
( \lambda_{i_2} - \lambda_{i_1} ) X_1 + 
( \lambda_{i_3} - \lambda_{i_2} ) X_2 + 
( \lambda_{i_4} - \lambda_{i_3} ) X_3 +
( \lambda_{i_1} - \lambda_{i_4} ) X_4 
=0
\end{equation}
which may have genus 1 or 2.  In the integer case, it may be of type $(3,1)$ or type $(2,2)$.  This will all depend on the ordering of the parts that is obtained from the ordering of the vertices on the 4-cycle.  For example, 
let $\vec{ \lambda} = ( 1,2,5,6) \in \mathbb{Z}^4$.  
The choice $i_1 = 1$, $i_2 = 2$, $i_3 =3$, and $i_4 = 4$ gives the type $(3,1)$ genus 1 equation 
\begin{equation}\label{a Behrend equation}
X_1 + 3X_2 +  X_3 = 5X_4.  
\end{equation}
{This equation corresponds to a 4-cycle whose vertices are in part 1, part 2, part 3, and then part 4.}  
The choice $i_1 = 1$, $i_2 =2$, $i_3 = 4$, and $i_4 = 3$
gives the type $(2,2)$ genus 2 equation 
\begin{equation}\label{a Sidon equation}
X_1 + 4X_2 = X_3 + 4X_4.
\end{equation}
Here a corresponding 4-cycle has vertices in part 1, 2, 4, then 3.
Lastly, up to symmetry, the choice 
$i_1 = 1$, $i_2 = 3$, $i_3 = 2$, $i_4 = 4$ gives
\begin{equation*}
    4X_1 + 4X_3 = 3X_2 + 5X_4,
\end{equation*}
a type $(2,2)$ genus 1 equation.  {The part ordering here is 1, 3, 2, then 4.}  

Returning to (\ref{a Sidon equation}), 
the trivial solutions $a + 4b - a - 4b=0$ with $a \neq b$ 
prevent us from using Proposition 
\ref{prop:consecutive edges}.  This situation can be avoided if the differences $\lambda_i - \lambda_j$ are all distinct.  Indeed, it was the corresponding differences 
\begin{center}
    $\lambda_2 - \lambda_1 =  1$,~~
    $\lambda_4 - \lambda_2 = 4$, ~~
    $\lambda_3 - \lambda_4 = -1$, ~~
    $\lambda_1 - \lambda_3 = -4$~~
    \end{center}
that gave (\ref{a Sidon equation}).  This leads to 
an idea, pioneered by Lazebnik and Verstra\"{e}te \cite{LV}, which is to choose the entries of $\vec{ \lambda}$ to be a Sidon set which we will define now.

\begin{definition}\label{def:sidon-set}
    Let $\Gamma$ be an abelian group where $+$ denotes the operation.   A set $A \subset \Gamma$ is a \emph{Sidon set} if whenever $a_1+a_2 = a_3 + a_4$ with $a_i \in A$, it must be the case that $\{a_1, a_2 \} = \{ a_3 , a_4 \}$.  Equivalently, $A$ is a Sidon set if and only if all of the nonzero differences $\{a_i - a_j : i \neq j \}$ are distinct. 
\end{definition}

To ensure that we have equations with non-zero coefficients, we choose $\vec{ \lambda} = ( \lambda_1 , \dots , \lambda_r)$ to have distinct entries, and to avoid trivial genus 2 4-cycles coming from equations like \eqref{a Sidon equation}, we choose $\vec{\lambda}$ to be a Sidon set.

Note that if $a,b,c$ are in a Sidon set of $S$ then the group element $a-b+c$ is not allowed to be in the Sidon set. Thus, a greedy argument gives that as long as $(r-1)^3 < |S|$, one may choose a Sidon set of size $r$.  Therefore, as long as $|S|$ grows, we can achieve this for arbitrary $r$. 

{Our construction using coding theory will have $S  = \mathbb{F}_q$ where $q$ can grow as function of $r$.  The Sidon set of size $r$ and distinct entries is the arithmetic structure from $\vec{ \lambda}$ needed for our proof of Theorem \ref{thm: graphs from codes}.  
Our construction used to prove Theorem \ref{thm: 456} will require $\vec{ \lambda}$ to have a third property}, which we define now.   

Let $\vec{ \lambda } = ( \lambda_1 , \lambda_2 , \dots , \lambda_r)$ be a vector whose entries are distinct integers that form a Sidon set.  Assume that 
\[
1 \leq \lambda_1 < \lambda_2 < \dots < \lambda_r.
\]
Let  $i_1,i_2,i_3,i_4$ be indices in $[r]$ with $i_1 \neq i_2$, $i_2 \neq i_3$, $i_3 \neq i_4$, $i_4 \neq i_1$ and let \begin{equation}\label{eq:special product}
\pi_{i_1 , i_2 , i_3 , i_4} = \coeff{2}{1} \coeff{3}{2} \coeff{4}{3} \coeff{1}{4}.  
\end{equation}
If the product $\pi_{i_1,i_2,i_3,i_4}$ is a positive integer that is not a square, let $s$ be the square-free part.  That is,
\[
\pi_{i_1,i_2,i_3,i_4} = s t^2
\]
where $t$ is a positive integer, and $s$ is the product of distinct primes, each with multiplicity one.  
The square-free integer $s$ is uniquely determined by 
the ordered 4-tuple $(i_1,i_2,i_3,i_4)$.  Let 
\[
\mathcal{S}( \vec{ \lambda} )\]
be the set of all square-free positive integers obtained in this way. 

\begin{definition}
Let $\vec{\lambda} = ( \lambda_1 , \lambda_2 , \dots , \lambda_r)$ be an integer vector whose entries form a Sidon set and $1 \leq \lambda_1 < \lambda_2 < \dots < \lambda_r$.  The vector $\vec{ \lambda}$ has the \emph{square products property} if there are infinitely many primes $p_1 < p_2 < 
\dots $ such that for each $p_i$, every integer in $\mathcal{S} ( \vec{ \lambda} )$ is a quadratic non-residue modulo $p_i$.
\end{definition}

\noindent
\textbf{Example:}  $\vec{ \lambda} = (1,28,33,36,43)$

The entries of $\vec{ \lambda}$ form a Sidon set 
in $\mathbb{Z}$.  There are three products of the form (\ref{eq:special product}) that are positive, but are not squares in $\mathbb{Z}$.  They are
\begin{eqnarray*}
    \pi_{1,4,2,3} = 
    (\lambda_1 - \lambda_4)( \lambda_4 - \lambda_2)
    (\lambda_2 - \lambda_3)( \lambda_3 - \lambda_1) 
    &=& (-35)(8)(-5)(32) = 7 \cdot 80^2, \\
   \pi_{1,5,2,3} =
   (\lambda_1 - \lambda_5)( \lambda_5 - \lambda_2) 
   (\lambda_2 - \lambda_3)(\lambda_3 - \lambda_1)
   & = &  (-42)(15)(-5)(32) =  7 \cdot 120^2, \\
    \pi_{2,4,5,3} = 
    ( \lambda_2 - \lambda_4) ( \lambda_4 - \lambda_5) 
    (\lambda_5 - \lambda_3) ( \lambda_3 - \lambda_2) 
    & = &     (-8)(-7)(10)(5) = 7 \cdot 20^2.  
\end{eqnarray*}
This shows that  $\mathcal{S} ( ( 1,28,33,36,43 )) = \{ 7 \}$.  Note that different choices of indices led to the same $s$, but 
$\mathcal{S} ( \vec{ \lambda} ) $ is a set, so repetition does not matter. 
The integer 7 is a quadratic non-residue modulo $p$ whenever $p$ is a prime that is $\pm 5, \pm 11, \pm 13$ modulo 28.  By Dirichlet's Theorem on primes in Arithmetic Progressions, there are infinitely many such primes.  
Therefore, $\vec{ \lambda}$ has the square products property.  

 It is worth noting that in 2020, Wei, Wang, and Schwartz \cite{WWS} applied Ruzsa's genus 2 method to 
 construct lattice packings in the context of error correcting codes. 
 In their proof they also required certain products to be either squares in $\mathbb{Z}$ or non-quadratic residues modulo $p$.  This same idea was used again by Wang 
 to give a new construction of constant-composition codes \cite{Wang}.    


\subsection{Avoiding cyles of length greater than 4}\label{section: longer cycles}

The methods of this paper will not be able to construct graphs of girth larger than $6$ without significant modification. For any $k\geq 6$, $\mathcal{H}(A,\vec{\lambda})$ will contain cycles of length $k$. We give two examples now that show this for $k\in \{6,7\}$. 
 
 Let $\lambda_{i} \neq \lambda_{j}$ be any two distinct entries of $\vec{ \lambda}$, and $\lambda = \lambda_i - \lambda_j$.  We can form a solution to 
\[
\lambda X_1  - \lambda X_2 + \lambda X_3 
- \lambda X_4 + \lambda X_5  - \lambda X_6 = 0 
\]
with $X_i \neq X_{i+1}$ by setting $X_1 = X_4 = a$, 
$X_2 = X_5 = b$, $X_3 = X_6 = c$ where $a,b,c$ are distinct elements of $A$. This gives $|R'||A|^3$ cycles of length $6$ no matter how $A$ and $\vec{\lambda}$ are chosen.  

For cycles of length 7, let $a,b,c$ be arbitrary elements of $A$ and $\lambda, \mu, \rho$ be arbitrary entries of $\vec{\lambda}$. Consider elements of $R$ $v_1,\cdots, v_7$ defined by 
\begin{alignat*}{3}
    v_1 & = &x_1 + a\lambda & = &x_7+c\lambda\\
    v_2 & = &x_2 + b\mu & = &x_1 + a\mu   \\
    v_3 & = &x_3 + a\rho & = &x_2 + b\rho\\
    v_4 & = &x_4 + c\lambda & = &x_3 + a\lambda\\
    v_5 & = &x_5 + a\mu & = &x_4 + c\mu \\
    v_6 & = &x_6 + b\rho & = &x_5 + a\rho \\
    v_7 & = &x_7 + c\mu  & =& x_6 + b\mu,
\end{alignat*}
and edges $e_1=e(x_1,a)$, $e_2=e(x_2, b)$, $e_3=e(x_3, a)$, $e_4=e(x_4, c)$, $e_5=e(x_5, a)$, $e_6=e(x_6, \rho)$, and $e_7=e(x_7, \mu)$ of $\mathcal{H}(A,\vec{\lambda})$. Note that once $a,b,c$ and $\lambda, \mu, \rho$ are chosen, then $v_1$ uniquely defines $v_2,\cdots, v_7$ and the $e_i$ form a cycle as long as all $x_i$ are distinct. In this case \eqref{a equation} reduces to 
\[
a(\lambda + \rho + \mu - \mu - \lambda - \rho) + b(\mu + \rho - \rho - \mu) + c(\lambda + \mu - \lambda - \mu) = 0,
\]
which holds trivially. Hence each vertex in $\mathcal{H}(A, \vec{\lambda})$ is in $\Omega(|A|^3)$ ``trivial" cycles of length $7$. For all $k\geq 8$, trivial cycles of length $k$ can similarly be found. 

\subsection{Avoiding cycles by genus type}\label{sec:avoiding cycles by genus type}

In some applications of this method, not every $k$-cycle in $\mathcal{H}( A , \vec{ \lambda} )$ is considered.  This includes Alon's notable paper on testing subgraphs \cite{alon} and Ge and Shangguan \cite{Ge} results on hypergraph Tur\'{a}n problems.  In these papers, the $k$-cycles that are avoided are those that have vertices in $k$ distinct parts.  When this restriction is made, 
the payoff can be that all forbidden cycles correspond to genus 1 equations, and so the critical obstacle that was unforeseen at the time that the claim from \cite{TV} was made does not cause an issue.


\section{Girth 5 and 6 with uniformity at least 7}\label{sec: codes}

\subsection{Tools from Coding Theory}

\label{subsec:tools-coding-theory} 

A linear code is a subspace of a vector space and hence can be described as the null space to a parity-check matrix. As a vector in the null space of a matrix naturally leads to an equation in terms of its columns, this will be useful for us. 

A subspace $\mathcal{C}$ of $\mathbb{F}_q^n$ of dimension $k$ is called an $[n, k]$ linear code. An $\ell\times n$ matrix $H$ for which $\mathcal{C}$ equals the nullspace of $H$ is called a parity-check matrix for $\mathcal{C}$.  The Hamming distance between any pair of vectors (called codewords) $\vec{v}_1, \vec{v}_2$ in $\mathcal{C}$  is $d(\vec{v}_1, \vec{v}_2)$, the number of positions in which they differ. The minimum distance of a code is defined as
\[ d(\mathcal{C})=\min_{\vec{v}_1,\vec{v}_2\in \mathcal{C}, \vec{v}_1\neq \vec{v}_2}\{d(\vec{v}_1, \vec{v}_2)\}.\]

Because the code is linear we have that $d(\vec{v}_1, \vec{v}_2) = d(\vec{v}_1- \vec{v}_2, \vec{0})$, and so distance can be understood by linear combinations of the columns of the parity check matrix which sum to $\vec{0}$. Recall that an $[n,k,d]_q$ code is an $[n,k]$ linear code in $\mathbb{F}_q^n$ with minimum distance $d$. We use the following folklore result which is standard in coding theory and relates the minimum distance of a code to the linear independence of subsets of the columns of a parity-check matrix for $\mathcal{C}$. 

\begin{proposition}[Folklore, see e.g. \cite{huffman2010fundamentals}]\label{prop: folklore}
   A linear code has minimum distance $d$ if and only if its parity check matrix has a set of $d$ linearly dependent columns, but any set of $d-1$ columns is linearly independent. 
\end{proposition}
   
We make use of Dumer's construction \cite{dumerdouble} of linear codes exceeding the BCH bound. 

\begin{theorem}[\cite{dumerdouble}, Theorem 7]\label{parity check dimensions}
  There exist codes of minimum distance at least 5 with parity-check matrix $H$, where $H$ has  
  $2m+\lceil \frac{m}{3}\rceil+1$ rows and $q^m$ columns, where $q$ is a power of an odd prime.
\end{theorem}

The parity-check matrix $H$ has the form $\displaystyle H= \begin{bmatrix} \tilde{H} \\\hline P \end{bmatrix}$
where $\tilde{H}$ is a parity-check matrix for an extended BCH code with parameters $[q^m, k= q^m-2m-1, d\geq 4]$, and $P$ is an $\lceil \frac{m}{3} \rceil \times q^m$ matrix whose columns come from the norm equation on $x=(\tau_1, \tau_2, \ldots, \tau_m)\in \mathbb{F}_q^m$. The construction is presented generally in \cite{dumerdouble};  we give the following example to illustrate. 

\begin{example} \label{example:F25-paritycheck-matrix} 
We construct a small example of the matrix $H$ described above, for $q=5$ and $m=2$. 
Use $\mathbb{F}_{25}\cong \mathbb{F}_5[x]/\langle x^2+2x+3\rangle$. Let $\alpha$ be a root of $x^2+2x+3$ over $\mathbb{F}_5$. Note that $\alpha$ is  primitive in $\mathbb{F}_{25}$ and generates $\mathbb{F}_{25}^*$. 
We identify $\mathbb{F}_{25}$ with $\mathbb{F}_5^2$ by the correspondence that takes $\alpha^n=\alpha X+Y$ to $(X,Y)\in \mathbb{F}_5^2$. In particular, $\alpha^0=1$ corresponds to $(0,1)\in \mathbb{F}_5^2$. Thus we obtain a matrix $H_1$ of the form 
\[ H_1=\begin{bmatrix}
    1&1&1&\cdots & 1&1 \\ 
    1 & \alpha&\alpha^2&\cdots & \alpha^{23}&\alpha^{24} \\ 
    1 & \alpha^2&\alpha^4&\cdots & \alpha^{46}&\alpha^{48}
\end{bmatrix}_{\mathbb{F}_{25}}\longrightarrow  \begin{bmatrix}
    0&0&0&\cdots & 0 &0\\ 
    1&1&1&\cdots & 1 &1\\
    0 & 1&3&\cdots & 3 &0\\ 
    1 & 0&2&\cdots & 1 &1\\
    0 & 3&4&\cdots & 3 &0\\
    1 & 2&2&\cdots & 1 &1\\
\end{bmatrix}_{\mathbb{F}_5^2}
\]

Removing the first redundant row gives a matrix $\tilde{H}$ with $2m+1$ rows and $q^m$ columns over $\mathbb{F}_q$. 

The next step is to append $P$, a $1\times 25$ matrix. For a given element $i$ of $\mathbb{F}_{25}$ indexing the $q^m$ columns, $P(i)$ is defined in \cite{dumerdouble} as a vector of length $\lceil \frac{m}{3}\rceil$ determined by the three-norm of triples within $i$. For this example, since $m=2$, $P$ can be defined as 
\[ P(i)=N_2(i)=i^{1+5} \in \mathbb{F}_5, \text{ for } i=1, \alpha, \ldots, \alpha^{23},0.\]
Therefore the matrix $P$ is
\[ P=\begin{bmatrix}
    1 & 3 & 4 & \cdots & 2& 0 
\end{bmatrix}.\]

Finally, the parity-check matrix with size as in Theorem~\ref{parity check dimensions} is
\[ H=\begin{bmatrix}
    1&1&1&\cdots & 1 &1\\
    0 & 1&3&\cdots & 3 &0\\ 
    1 & 0&2&\cdots & 1 &1\\
    0 & 3&4&\cdots & 3 &0\\
    1 & 2&2&\cdots & 1 &1\\
    1 & 3 & 4 & \cdots & 2& 0 
\end{bmatrix}_{6\times 25}.\]


    
\end{example}

\begin{theorem}[\cite{dumerdouble}, Cor. 6] \label{d6 parity check dimensions}
 The code $V_6^m$ 
 has minimum distance at least 6 and length $q^{\lfloor 5(m-1)/6\rfloor}$, for $m$ even. There is a parity check matrix for $V_6^m$ with $\frac{5}{2}m$ rows and $q^{\lfloor 5(m-1)/6\rfloor}$ columns.     
\end{theorem}

The code $V_6^m$ is a punctured BCH code (see e.g. \cite{macwilliams1977theory}) designed to   
retain certain  codeword positions that result in the parameters given in Theorem~\ref{d6 parity check dimensions}. 


\begin{remark}
    The constructions given for Theorems~\ref{parity check dimensions} and \ref{d6 parity check dimensions} 
    given in \cite{dumerdouble} have redundant rows. Retaining only a linearly independent set of rows  yields a parity-check matrix with the sizes given.
\end{remark}

\subsection{Ruzsa-Szemer\'{e}di Hypergraphs from Dumer's BCH Codes}

  Our next task is to prove Theorem \ref{thm: graphs from codes}. Let $\mathcal{C}$ be a linear code of distance $d\in \{5,6\}$ over $\mathbb{F}_q$ and let $A$ be the set of columns of its parity check matrix. Let $\vec{\lambda}$ be a Sidon set of size $r$ in $\mathbb{F}_q$. We assume that $q$ is fixed but is large enough so that it contains such a Sidon set. Let $\mathcal{H} = \mathcal{H}(A,\vec{\lambda})$. Here we are using $S = \mathbb{F}_q$ and $R = R' = \mathbb{F}_q^t$ where $t$ is the number of rows in the parity check matrix. We will show that $\mathcal{H}$ has girth at least $d$.

\subsection{Proof of girth}\label{sec: coding theory girth}
By Proposition \ref{prop: folklore}, we have that there is no nontrivial linear combination of less than $d$ vectors in $A$ which sum to $\vec{0}$. Proposition \ref{prop: linear} gives us that $\mathcal{H}$ has no $2$-cycles. We now consider cycles of length $3$, $4$, and $5$. 

\medskip \noindent {\em No $3$-cycles:} Assume that there is a $3$-cycle with vertices $v_1, v_2, v_3$ and edges $E_1$, $E_2$, $E_3$ with the same notation as above. Then \eqref{a equation} gives that 
\[
(\lambda_{i_1} - \lambda_{i_2})a_1 + (\lambda_{i_2} - \lambda_{i_3})a_2 + (\lambda_{i_3} - \lambda_{i_1})a_3 = 0.
\]
Because we have that  $i_3\not = i_1$ and $i_\ell \not = i_{\ell+1}$ for all $1\leq \ell \leq 2$, all of the coefficients of the $a_i$ are nonzero. Since sets of at most $3$ vectors in $A$ cannot be in nontrivial linear combination summing to $\vec{0}$, we must have that $a_1 = a_2 = a_3$. But this implies that $x_1 = x_2 = x_3$ and so $E_1 = E_2 = E_3$, a contradiction.

\medskip \noindent {\em No $4$-cycles:} Assume there is a $4$-cycle with vertices $v_1, v_2 , v_3, v_4$ and edges $E_1,E_2 , E_3, E_4$. Then by \eqref{a equation}, we have 
\[
(\lambda_{i_1} - \lambda_{i_4})a_1 + (\lambda_{i_2} - \lambda_{i_1})a_2 + (\lambda_{i_3} - \lambda_{i_2})a_3 + (\lambda_{i_4} - \lambda_{i_3})a_4 = 0.
\]

Because we have that  $i_4\not = i_1$ and $i_\ell \not = i_{\ell+1}$ for all $1\leq \ell \leq 3$, all of the coefficients of the $a_i$ are nonzero. By the condition on $A$ this must again be a trivial linear combination. Note that if $a_i = a_{i+1}$ for any $i$ (modulo $4$), we have that $E_i = E_{i+1}$, a contradiction. Hence we must have that $a_1 = a_3$ and $a_2 = a_4$, and furthermore that $\lambda_{i_1} - \lambda_{i_4} + \lambda_{i_3} - \lambda_{i_2} = 0$. Since $\vec{\lambda}$ is a Sidon set, we must have that $\{i_1, i_3\} = \{i_2, i_4\}$. But then we have that $i_\ell = i_{\ell+1}$ for some $\ell$ (modulo $4$), a contradiction. 

\medskip \noindent {\em No $5$-cycles:} Assume there is a $5$-cycle with vertices $v_1,\cdots, v_5$ and $E_1,\cdots, E_5$. Then by \eqref{a equation}, we have 
\[
(\lambda_{i_1} - \lambda_{i_5})a_1 + (\lambda_{i_2} - \lambda_{i_1})a_2 + (\lambda_{i_3} - \lambda_{i_2})a_3 + (\lambda_{i_4} - \lambda_{i_3})a_4 + (\lambda_{i_5 }- \lambda_{i_4})a_5 = 0.
\]
Because we have that  $i_5\not = i_1$ and $i_\ell \not = i_{\ell+1}$ for all $1\leq \ell \leq 4$, all of the coefficients of the $a_i$ are nonzero. Now, if $d\geq 6$, then there can be no nontrivial combination of at most $5$ vectors in $A$ summing to $\vec{0}$, and hence the equation must be trivial. This implies that at least $3$ of the $a_i$ must be equal, which implies that $i_\ell = i_{\ell+1}$ for some $\ell$ (modulo $5$), a contradiction.

\subsection{Proof of Theorem \ref{thm: graphs from codes}}

{Fix an integer $r \geq 3$.}
First let $d=5$. Let $q$ be a prime large enough so that there exists a Sidon set $ \vec{\lambda} = \{\lambda_1,\cdots, \lambda_r\}$ in $\mathbb{F}_q$. For example one may choose greedily as explained in Section \ref{sec: setup} and find such a set if $(r-1)^3 < q$.  Let $m$ be the largest integer satisfying 
\[
r\cdot q^{2m+\lceil \frac{m}{3}\rceil} \leq N. 
\]
We have that $m = \frac{3}{7}\log_q N - O_r(1)$. By Theorem \ref{parity check dimensions}, there is a linear code of distance $5$ with a parity check matrix that has $2m + \lceil \frac{m}{3}\rceil$ rows and $q^m$ columns. Define $\mathcal{H}_5 = \mathcal{H}(A, \vec{\lambda})$ where $A$ is the set of columns of this parity check matrix. Then this hypergraph on at most $N$ vertices is $|A| = q^m$ regular and since $q$ is depending only on $r$, we have that 
\[
q^m = q^{\frac{3}{7}\log_q N - O_r(1)} \geq c_r N^{3/7}.
\]

By Section \ref{sec: coding theory girth}, $\mathcal{H}_5$ has girth $5$, implying this case of Theorem \ref{thm: graphs from codes}. 

Similarly, for $d=6$, let $m$ be the largest even integer with $r\cdot q^{5m/2} \leq N$, so $m = \frac{2}{5} \log_q N - O_r(1)$. By Theorem \ref{d6 parity check dimensions}, there is a linear code with distance $6$ which has a parity check matrix with $\frac{5}{2}m$ rows and $q^{\lfloor 5(m-1)/6\rfloor}$ columns. Define $\mathcal{H}_6 = \mathcal{H}(B, \vec{\lambda})$ where $B$ is the set of columns of this matrix. Then $\mathcal{H}_6$ is a hypergraph on at most $N$ vertices which is $|B| = q^{\lfloor 5(m-1)/6\rfloor}$ regular, and thus we have average degree at least 
\[
q^{\lfloor 5(m-1)/6\rfloor} = c'_r N^{1/3}.
\]

By Section \ref{sec: coding theory girth}, $\mathcal{H}_6$ has girth $6$, proving the theorem.

\section{Girth 5 and uniformities 4, 5, and 6}\label{sec: ruzsa method}

\subsection{Tools from Number Theory}

Now we discuss our approach using number theory.  The general strategy is the same in that we use Proposition \ref{prop:cycle equations} and avoid cycles by avoiding solutions to \eqref{a equation}. 
In this construction, we will use the results and methods of Ruzsa \cite{Ruzsa}.  We note that we work with $R = \mathbb{F}_p \times \mathbb{F}_p$ whereas Ruzsa phrased this construction over the integers, but one can easily see that Ruzsa is ``actually" working in $\mathbb{F}_p \times \mathbb{F}_p$.  We phrase it this way because it makes the details clearer in our opinion.

Using this method is more involved than with a code of minimum distance $d$ where by Proposition \ref{prop: folklore} we automatically avoid all nontrivial equations with at most $d-1$ variables. Before getting into the details, we outline the strategy and make a special note of where an obstacle arises while using Ruzsa's ideas. We show how to overcome this obstacle when $r\in \{4,5,6\}$, {and it is an open problem if this same approach is possible for larger $r$.}   

The initial thought behind the claimed lower bound on 
$\textup{ex}_r ( N , \mathcal{C}_{< 5} )$ from \cite{TV} is that if the entries of $\vec{ \lambda}$ are distinct
and form a Sidon set, then a natural modification of the construction from \cite{Ruzsa} would give a set $A$ of size $N^{1/2 - o(1)}$ that avoids non-trivial solutions to 2, 3, and 4-cycle equations determined by $\vec{\lambda}$. 
 As discussed in Section \ref{sec: setup}, this requires avoiding families of genus 1 and genus 2 equations.  The known methods to construct sets avoiding solutions to genus 1 and genus 2 equations are very different.  Ruzsa's paper \cite{Ruzsa} makes contributions to both. For the Behrend equation, Ruzsa observes (a similar modification appears in \cite{alon}) that Behrend's original argument extends to any equation of the form 
\[
\beta_1 X_1 + \beta_2 X_2 + \dots + \beta_k X_k = 
(\beta_1 + \beta_2 + \dots + \beta_k)X_{k+1}.
\]
where the $\beta_i$'s are positive integers. This part of Ruzsa's method goes through easily for us as well.
By Proposition \ref{prop: linear}, we will not see $2$-cycles in our construction. Cycles of length $3$ correspond to solutions to an equation in $3$ variables with coefficients bounded by a constant depending only on the $\lambda_i$'s.  Because such a cycle must be genus 1 we can use Lemma \ref{lemma:multiple Behrend} to avoid them. 

The more difficult part arises when we must deal with 
multiple genus 2 equations coming from different 4-cycles.  
 While \cite{Ruzsa} gives a method for constructing a set with 
 many elements and no solution to a single genus 2 equation of type $(2,2)$, 
 there is a subtlety that surfaces when one tries to extend this method 
 to families of these types of equations. The subtlety arises as follows.

Cycles of length $4$ correspond to an equation in $4$ variables over $\mathbb{F}_p^2$.  This equation is equivalent to a system of two equations in 4 variables 
over $\mathbb{F}_p$. Following \cite{Ruzsa}, algebraic manipulation shows that if a solution to the system exists, then one can eliminate a variable and obtain 
a solution to a different equation over $\mathbb{F}_p$ that is in 3 variables.   
Crucially though, the coefficients in this new equation are determined in part by taking a square root in $\mathbb{F}_p$ and hence, {\em we do not have control of their size as integers}.  {Thus, we are led to consider equations} of the form $aX + bY = (a+b)Z$ where $a,b$ can be any field elements. To our knowledge, it is not known whether or not a set of size $p^{1-o(1)}$ avoiding such an equation may be found in $\mathbb{F}_p$.  This is exactly the gap in the claim from \cite{TV} that a girth $5$ hypergraph with $N^{3/2-o(1)}$ edges exists for any uniformity. 
Howevever, one of our main contributions is to show that with the square products property, it is possible to overcome this issue when $r\in \{4,5,6\}$, proving Theorem \ref{thm: 456}.  We now proceed with the details.

As introduced in Section \ref{sec: setup} our approach using number theory requires the 
entries of $\vec{ \lambda}$ to be a Sidon set with the square products property.  
The largest such set we were able to find has 6 elements.  It was found using a computer search by David Davini \cite{David}.  

\begin{proposition}[Davini \cite{David}]\label{prop:size 6}
    The entries of the vector 
    \[
    \vec{\lambda} = ( 1,35,161,170,251,545)
    \]
    form a Sidon set of integers of size 6 with the square products property.  
        \end{proposition} 

It can be checked that the entries of $\lambda$ form a Sidon set and $\mathcal{S}( \vec{ \lambda} ) = \{3 \cdot  7 , 3 \cdot 17 \}$ (see the Appendix).  By the Law of Quadratic Reciprocity, if $p$ is a prime that is 
\begin{itemize}
    \item 5 or 7 modulo 12, then 3 is a quadratic non-residue modulo $p$,
    \item 1, 3, 9, 19, 25, or 27 modulo 28, then 7 is a quadratic residue modulo $p$,
    \item 1, 2, 4, 8, 9, 13, 15, or 16 modulo 17, then 17 is a quadratic residue modulo $p$.  
\end{itemize}
For such a prime, both $3 \cdot 7$ and $3 \cdot 17$ will be quadratic non-residues modulo $p$.  A prime $p$ that satisfies solves the system  
\begin{center}
    $x \equiv 5 ( \textup{mod}~12)$, ~~~$x \equiv 1 ( \textup{mod}~28)$, ~~~$x \equiv 1 ( \textup{mod}~17)$ 
\end{center}
will satisfy $p \equiv 953 ( \textup{mod}~1428)$.  By Dirichlet's Theorem on primes in Arithmetic Progressions, there are infinitely many such primes $p$.  
For such a prime $p$, both $3 \cdot 7$ and 
$3 \cdot 17$ are quadratic non-residues modulo $p$, so 
$\vec{\lambda}$ has the square products property.  

We will need our sequence of primes to be dense enough so that given $N$, we can choose a value of $rp^2$ that is close to $N$ and $rp^2 \leq N$.  To prove such a sequence exists, we use a quantitative form of the Prime Number Theorem in Arithmetic Progressions.

\begin{lemma}\label{lemma:PNT lemma}
Let $a$ and $m$ be positive integers that are relatively prime, and let $\epsilon >0$.  There is an $M_0$ such that for all $M > M_0$, there is a prime $p $ with $p \equiv a ( \textup{mod}~m)$ and 
$(1 - \epsilon) M \leq p \leq M$.   
  \end{lemma}
  \begin{proof}
Let $\pi_{a,m} (x)$ be the number of primes $p $ with $p \equiv a ( \textup{mod}~m)$ and $p \leq x$.  Let $\textrm{Li}(x) = \int_2^x \frac{dt}{ \log t}$ and $\phi$ be the Euler phi function.  Using a result of 
\cite{Bennett}, for sufficiently large $x$ we have
\[
\left| \pi_{a,m}(x)  - \frac{\textrm{Li}(x) }{ \phi (m) } \right| 
< \frac{ C x}{ ( \log x )^2}
\]
where $C$ is a positive constant depending only on $m$.  Thus,
\begin{eqnarray*} 
\pi_{ a , m }( M )  - \pi_{a , m}( ( 1 - \epsilon ) M) & \geq & 
\frac{1}{ \phi (m) } \int_{ ( 1 - \epsilon )M }^{M}  \frac{dt}{ \log t} 
- \frac{2CM}{ ( \log ( ( 1 - \epsilon ) M )^2  } \\ 
& > & 
\frac{ \epsilon M}{ (m - 1) \log ( M ) } - \frac{3 cM}{ ( \log ( ( 1 - \epsilon ) M )^2}.
\end{eqnarray*}
 For large enough $M$, depending on $\epsilon$ and $m$, this quantity is positive.  Therefore, 
 there is a prime $p$ with $( 1- \epsilon ) M \leq p \leq M$ and 
 $p \equiv a ( \textup{mod}~m)$.  
    \end{proof}
    \medskip


\subsection{Ruzsa-Szemer\'{e}di Hypergraphs from Square Product Sidon Sets}\label{sec:4.2}

Let $r \in \{4,5,6 \}$.  
Let $\vec{\lambda} = ( \lambda_1 , \dots , \lambda_r) $ 
be a vector whose entries are positive integers that 
form a Sidon set with the square products property.  For $r=6$ such a set exists by 
Proposition \ref{prop:size 6}.  
Any subset of a Sidon set with the square products property is also a Sidon set with the square products property.  Thus, for $r \in \{4,5 \}$ one can restrict to the first $r$ entries of 
$\vec{ \lambda} = ( 1,35,161,170,251,545)$ from Proposition \ref{prop:size 6} to obtain the needed set.  

By Lemma \ref{lemma:multiple Behrend} with $C= 2 \lambda_r^2$ and $\ell = 3$, there are positive real numbers $M$, $\tau$, and $\gamma$ such that the following holds.  For any prime $p > M$, there is a set $B \subset \mathbb{F}_p$ with 
$|B| > \gamma p e^{ - \tau \sqrt{ \log p } } $,
and $B$ has only trivial solutions to each 
equation in $\mathcal{E}_{ 2 \lambda_r^2 , 3}$.  
Consider the sequence of all primes 
\[
 M < p_1 < p_2 < p_3 < \dots 
\]
where $p_i \equiv 953 ( \textrm{mod}~1428)$ for all $i$, and $p_1$ is the smallest such prime that is larger than $M$.   
Write $\mathcal{P}$ for this sequence.  
For a prime $p \in \mathcal{P}$, let 
\[
A = \{ (x,x^2) : x \in B \} \subset \mathbb{F}_p \times \mathbb{F}_p.
\]

 Define $\mathcal{H}( A , \vec{ \lambda} )$ to be
 the resulting $r$-uniform hypergraph.  This hypergraph 
 has $r p^2$ vertices and $p^2 |B| >  \gamma p^3 e^{ - \tau \sqrt{ \log p } }$ edges.  In Section \ref{sec: proof of girth number theory} we prove that $\mathcal{H}( A , \vec{ \lambda} )$ has girth $5$.


\subsection{Proof of girth}\label{sec: proof of girth number theory}

\medskip \noindent {\em No $3$-cycles:} Assume that there is a 3-cycle with edges 
$e(w_1 , x)$, $e(w_2,y)$, $e(w_3,u)$.  
By Proposition \ref{prop:cycle equations} there are indices $1 \leq i_1,i_2,i_3 \leq r$ with $i_1 \neq i_2$, $i_2 \neq i_3$, and $i_3 \neq i_1$, and 
\[
( \lambda_{i_2} - \lambda_{i_1} ) (x,x^2)
+ 
(\lambda_{i_3} - \lambda_{i_2} ) (y,y^2)
+ 
(\lambda_{i_1} - \lambda_{i_3} ) (u,u^2) 
=(0,0).  
\]
The first coordinate implies that there are elements $x,y,u \in B$ that form a solution to  
\[
\coeff{2}{1} X + \coeff{3}{2} Y + \coeff{1}{3}U = 0.
\]
This is a type $(2,1)$ equation whose coefficients 
are bounded by $2 \lambda_r^2$.  By Lemma \ref{lemma:multiple Behrend} it must be 
the case that $x=y=u$.  By Proposition \ref{prop:consecutive edges} the edges $e(w_1,x)$ and $e(w_2,y) $ are the same since $x = y$ and they share a vertex.  This is a contradiction so $\mathcal{H}(A , \vec{ \lambda })$ has no 3-cycles.  

\medskip \noindent {\em No $4$-cycles:} Assume that there is a 4-cycle with edges $e(w_1 , x)$, $e(w_2,y)$, $e(w_3,u)$, and $e(w_4,v)$. 
%
By Proposition \ref{prop:cycle equations} there are indices $1 \leq i_1,i_2,i_3,i_4 \leq r$ where $i_1 \neq i_2$, $i_2 \neq i_3$, $i_3 \neq i_4$, $i_4 \neq i_1$, and 
\[
( \lambda_{i_2} - \lambda_{i_1} ) (x,x^2)
+ 
(\lambda_{i_3} - \lambda_{i_2} ) (y,y^2)
+ 
(\lambda_{i_4} - \lambda_{i_3} ) (u,u^2) 
+ 
(\lambda_{i_1} - \lambda_{i_4} ) ( v,v^2) 
=(0,0).  
\]
From the first coordinate we have the equation 
  \begin{equation}\label{eq:first coordinate}
 \coeff{2}{1} x + \coeff{3}{2} y + \coeff{4}{3} u + \coeff{1}{4} v  = 0,
 \end{equation}
 and from the second,
 \begin{equation}\label{eq:second coordinate}
  \coeff{2}{1} x^2 + \coeff{3}{2} y^2 + \coeff{4}{3} u^2 + \coeff{1}{4} v^2  = 0.
 \end{equation}
If the product $\pi_{i_1,i_2,i_3,i_4} = 
 \coeff{2}{1} \coeff{3}{2} \coeff{4}{3} \coeff{1}{4}$ 
 is negative, then 
   \begin{equation*}
 \coeff{2}{1} X + \coeff{3}{2} Y + \coeff{4}{3} U + \coeff{1}{4} V  = 0
 \end{equation*}
 is a type $(3,1)$ equation with coefficients bounded 
 by $2 \lambda_r^2$.  By Lemma \ref{lemma:multiple Behrend} and (\ref{eq:first coordinate}), it must be the case that $x = y = u = v$.  This implies $e(w_1,x)$ and $e(w_2,y)$ are the same edge, a contradiction.

 Assume now that the product $\pi_{i_1,i_2,i_3,i_4}$ is a  positive integer.  Then exactly two of $\coeffnp{2}{1}$, $\coeffnp{3}{2}$, $\coeffnp{4}{3}$, $\coeffnp{1}{4}$ are positive integers.  Without loss of generality, the two cases we consider are  $\coeffnp{2}{1}, \coeffnp{3}{2} > 0$ and $\coeffnp{2}{1}, \coeffnp{4}{3} > 0$.
   
  \medskip
 \noindent
 \textbf{Case 1:} $\coeffnp{2}{1} > 0$, $\coeffnp{3}{2} > 0$
 
  \smallskip
For convenience of notation, let $a = \coeffnp{2}{1}$, 
$b = \coeffnp{3}{2}$, $c = - ( \coeffnp{4}{3} )$, and $d = - ( \coeffnp{1}{4} )$.
Then $a,b,c,d$ are all positive integers with $a+b = c+ d$.  With this notation, 
(\ref{eq:first 
 coordinate}) and (\ref{eq:second coordinate}) become
 \begin{equation}\label{eq:7.9}
 a x + by = c u + dv
 \end{equation}
 and 
 \begin{equation}\label{eq:7.10}
 ax^2 + by^2 = cu^2 + dv^2.
 \end{equation}
Following Ruzsa's proof of Theorem 7.3 in \cite{Ruzsa}, multiply (\ref{eq:7.10}) through by $a+b=c+d$ to get
\[
a^2 x^2 + ab y^2 + b a x^2 + b^2 y = c^2 u^2 + c d v^2 + dc u^2 + d^2 v^2.
\]
Subtracting the square of (\ref{eq:7.9}) leads to
 \begin{equation}\label{eq:7.13}
 ab ( x- y )^2 = cd ( u - v)^2.
 \end{equation}
 
  If $x = y$, then since $ab \neq 0$ and $cd \neq 0$, (\ref{eq:7.13}) implies $u  = v$.  We then substitute into (\ref{eq:7.9}) to get $(a + b) x = (c + d) u $.  Recalling 
  that $a+b=c+d \neq 0$, we may cancel to get $x = u$, which then implies 
  $x=y=u=v$.  Since $e(w_1,x)$ and $e(w_2,y)$ have a vertex in common and $x = y$, Proposition \ref{prop:consecutive edges} implies these edges are the same, a contradiction.    

Now assume $x \neq y$ and $u \neq v$.  Multiplying (\ref{eq:7.13}) through by $cd$ yields 
 \begin{equation}\label{eq:quadratic abcd}
 abcd ( x - y)^2 = (cd)^2 ( u - v)^2.
 \end{equation}
Let $$\pi_{i_1,i_2,i_3,i_4}  = s t^2$$ where $s$ and $t$ are positive integers and $s$ is square free.  Note 
that
 \[
 st^2 = \coeff{2}{1} \coeff{3}{2} \coeff{4}{3} \coeff{1}{4} = abcd.  
 \]
From this equation, we obtain 
\begin{equation}\label{eq:s t squared}
st^2 ( x - y)^2 = (cd)^2 ( u -v)^2 .  
\end{equation}

Since $x-y \neq 0$ and $u - v \neq 0$, (\ref{eq:s t squared}) implies that $s$ is a quadratic residue modulo $p$.  Write $[ \sqrt{s} ]_p$ for one of the two solutions to 
$X^2 = s (\textup{mod}~p)$.  Using this notation, (\ref{eq:s t squared}) can be rewritten as 
\[
\left(  [ \sqrt{s} ]_p t (x - y ) - cd ( u -v) \right) 
\left( [ \sqrt{s} ]_p t (x - y ) + cd ( u - v) \right) = 0.
\]
Suppose first that $[ \sqrt{s} ]_p t (x - y ) - cd ( u -v) = 0$.   Multiplying 
(\ref{eq:7.9}) by $c$ and then adding  
to $[ \sqrt{s} ]_p tx - [ \sqrt{s} ]_p ty - cd u + cdv = 0$ gives 
\[
( [ \sqrt{s} ]_p t  + ca ) x + ( cb - [ \sqrt{s} ]_p t )y - c ( c + d)u = 0.
\]
This is an equation over $\mathbb{F}_p$ whose coefficients are 
$[ \sqrt{s} ]_p t  + ca$, 
$cb - [ \sqrt{s} ]_p t$, and $- c(c + d)$ where $x,y,u \in B$ form a solution to 
\begin{equation}\label{eq:bad 1}
\left(  [ \sqrt{s} ]_p t  + ca \right) X + 
\left( cb   - [ \sqrt{s} ]_p t \right) Y = c (c + d)U.
\end{equation}

Now suppose that $ [ \sqrt{s} ]_p t (x - y ) + cd ( u - v)  = 0$.  Multiplying (\ref{eq:7.9}) by $d$ and adding to $[ \sqrt{s} ]_p tx - [ \sqrt{s} ]_p ty = -cdu + cdv$ gives
\[
( [ \sqrt{s} ]_p t  + da ) x + ( db - [ \sqrt{s} ]_p t )y  = d(c+d)v.
\]
Therefore, $x,y,v \in B$ form a solution to the equation 
\begin{equation}\label{eq:bad 2}
( [ \sqrt{s} ]_p t  + da ) X + ( db - [ \sqrt{s} ]_p t )Y  = d(c+d)V.
\end{equation}

\bigskip
\noindent
\textbf{Subcase 1:} $s = 1$

\medskip
  If $s = 1$, then $abcd = t^2$ where $t$ is a positive integer with $1 \leq t \leq \lambda_r^2$.  Either the elements $x,y,u$ or the elements $x,y,v$ form a solution to the 
 type $(2,1)$ equation 
 \[ (t + ca ) X + ( cb - t)Y = c ( c + d)  U ~~~ \mbox{or} ~~~ 
(  t  + da ) X + ( db - t )Y  = d(c+d)V,
 \]
respectively.  The coefficients in both of these equations have absolute value at most $2 \lambda_r^2$.  Hence, we can apply Lemma \ref{lemma:multiple Behrend} to get either $x = y = u$ in the first case, or $x=y=v$ in the second.  This is a contradiction since we have assumed that $x \neq y$.       
 
\bigskip
\noindent
\textbf{Subcase 2:} $s > 1$

\medskip

Recalling that $st^2 = \coeff{2}{1} \coeff{3}{2} \coeff{4}{3} \coeff{1}{4}$, we see
that $s$ is the square-free part of a product of differences of elements of $\vec{ \lambda}$.  
By definition, $s \in \mathcal{S} ( \vec{ \lambda})$ and in this case, 
$\mathcal{S} ( \vec{ \lambda}) = \{ 3 \cdot 7 , 3 \cdot 17 \}$.  
Each prime $p$ has been chosen so that 3 is a non-quadratic reside, while 7 and 17 are quadratic residues.  Therefore, $3 \cdot 7$ and $ 3 \cdot 17$ are not quadratic residues modulo $p$. 
Since $s$ not a quadratic residue modulo $p$, we have a contradiction with 
(\ref{eq:s t squared}).

 
 \begin{remark}\label{remark: subtlety}
     If we do not assume that $\vec{\lambda}$ has the square products property, then it may be the case that $s\not=1$ and $s$ is a quadratic residue in $\mathbb{F}_p$. In this case we still obtain the pair of equations (\ref{eq:bad 1}) and (\ref{eq:bad 2}). Both are type 
     $(2,1)$ equations if considered as equations over the integers, but the coefficients will be a function of the square root of $s$ in $\mathbb{F}_p$.  Now unlike sums and products of the elements of $\vec{ \lambda}$, 
     the square root of $s$ modulo $p$ may not be bounded by a constant that is independent of the $p$ (and indeed in this case the square root considered as an integer must be at least $\sqrt{p}$).  Thus, following this particular approach we cannot apply Lemma \ref{lemma:multiple Behrend} because we do not have a uniform bound on all of the coefficients that appear in 3 variable equations that come from reducing the 4 variable systems of 2 equations coming from 4-cycles.
 \end{remark}
\medskip

\noindent
\textbf{Case 2:} $\coeffnp{2}{1} > 0$, $\coeffnp{4}{3} > 0$

\smallskip

Let $a = \coeffnp{2}{1}$, $b = \coeffnp{4}{3}$, $c = -  \coeff{3}{2} $, and $d  = -  \coeff{1}{4} $ so 
$a,b,c,d$ are positive integers and $a  + b = d + c$.  Equations (\ref{eq:first 
 coordinate}) and (\ref{eq:second coordinate}) are now 
 \begin{equation}\label{eq:7.9 Case 2}
 a x + bu = c y + dv
 \end{equation}
 and 
 \begin{equation}\label{eq:7.10 Case 2}
 ax^2 + bu^2 = cy^2 + dv^2.
 \end{equation}
As in Case 1, we can obtain
 \begin{equation}\label{eq:7.13 Case 2}
 ab ( x- u )^2 = cd ( y - v)^2.
 \end{equation}
If $x = u$, then (\ref{eq:7.13 Case 2}) implies $y =v$.  Substituting into (\ref{eq:7.10 Case 2}) 
gives $(a+b)x = (c + d )y$ so that $x = y$.  This is a contradiction since $x = y$ implies that 
$e(w_1,x) = e(w_2,y)$.  
Assume that $x \neq u$ and $y \neq v$.  Multiplying (\ref{eq:7.13 Case 2}) through by $cd$ yields 
 \[
 abcd ( x - u)^2 = (cd)^2 ( y - v)^2.
 \]
From this point forward, the argument is similar to that of Case 1 starting with equation 
(\ref{eq:quadratic abcd}).  One writes $abcd$ as $st^2$ where $s$ and $t$ are integers and
$s$ is square free to get 
 \[
 s t^2 ( x- u )^2 = (cd)^2 ( y - v)^2.
 \]
 Since $x \neq u$ and $y \neq v$, this equation implies that $s$ is a quadratic residue modulo $p$ and we may write 
\[
\left(  [ \sqrt{s} ]_p t (x - u ) - cd ( y -v) \right) 
\left( [ \sqrt{s} ]_p t (x - u ) + cd ( y - v) \right) = 0.
\]
The remaining steps are similar to Case 1, and Case 2 leads to another contradiction.


\subsection{Proof of Theorem \ref{thm: 456}}

Let $r \in \{4,5,6 \}$.  For $r = 4,5,6$, we take $\vec{ \lambda}$ to be $(1,35,161,170)$, $(1,35,161,170,251)$, $(1,35,161,170,251,545)$, respectively.  Let $\epsilon > 0$.   
By Lemma \ref{lemma:PNT lemma} there is an $N_r$ such that 
for all $N > N_r$, there is a prime $p$ with $p \equiv 953 ( \textup{mod}~1428)$ and 
$(1 - \epsilon ) \sqrt{N/r} \leq p \leq 
\sqrt{N / r}$.   
This latter inequality can be rewritten as $(1 - \epsilon)^2 N \leq r p^2 \leq N$.  
We also choose $N$ large enough so that the lower bound 
$(1 - \epsilon ) \sqrt{N/r} \leq p$ implies that we can choose $B$ according to Section \ref{sec:4.2}: $B$ has only trivial solutions to each equation in $\mathcal{E}_{ 2 \lambda_r^2 , 3}$ and 
$|B| > \gamma p e^{ - \tau  \sqrt{ \log p} }$.  Here $\gamma$ and $\tau$ are positive constants depending only on $r$ and the entries of $\vec{ \lambda}$.  Let
$R = R' = \mathbb{F}_p^2$, $S = \mathbb{Z}$, 
\[
A  = \{ ( x,x^2) : x \in B \}
\]
and $\mathcal{H} ( A , \vec{ \lambda} )$ be the corresponding hypergraph.  
This hypergraph has $rp^2$ vertices and girth 5 by Section \ref{sec: proof of girth number theory}.  Since $(1 - \epsilon )^2 N \leq rp^2 \leq N$, $\mathcal{H}(A , \vec{ \lambda})$ has at least 
\[
p^2 |A| > \gamma p^3 e^{  - \tau \sqrt{ \log p } } \geq \frac{\gamma(1-\epsilon)^3}{r^{3/2}}N^{3/2} e^{-\tau \sqrt{\log(\sqrt{N/r})}} = N^{3/2-o(1)}
\]
edges.


\section{Coding theory bounds from graphs}\label{sec: sphere packing}


In this section we go the other direction and use graph theoretic results to imply bounds on codes, proving Theorem \ref{distance 6 code upper bound}. To prove this we use the following recent theorem of Conlon, Fox, Sudakov, and Zhao \cite{conlon}.

\begin{theorem}[\cite{conlon}, Corollary 1.10]\label{girth 6 upper bound}
    For $r\geq 3$, every $r$-uniform hypergraph on $N$ vertices with girth $6$ has $o(N^{3/2})$ edges.
\end{theorem}

\begin{proof}[Proof of Theorem \ref{distance 6 code upper bound}]
    Let $\mathcal{C}$ be an $[n,k,6]_q$ code and let $H$ be a parity check matrix for it. Let $A$ be the set of the columns of $H$, $A\subset \mathbb{F}_q^{n-k}$, and let $\vec{\lambda}=\{\lambda_1 ,\lambda_2, \lambda_3\}\subset \mathbb{F}_q$ be a Sidon set of size $3$, noting that one exists because $q\geq 7$ and is odd. Define the graph $G=\mathcal{H}(A,\vec{\lambda})$. This is a graph on $N = 3\cdot q^{n-k}$ vertices with $|A|q^{n-k} = n q^{n-k}$ edges. By Section \ref{sec: coding theory girth}, $G$ has girth $6$, and so by Theorem \ref{girth 6 upper bound} it must have $o(N^{3/2})$ edges. Therefore, we have that 
    $nq^{n-k} = o\left( 3^{3/2}q^{\frac{3}{2}(n-k)}\right)$
    or equivalently that 
    \[
    \lim_{n\to \infty } \frac{q^{\frac{1}{2}(n-k)}}{n} = 0.
    \]
    Therefore $q^{\frac{1}{2}(n-k)} = \omega(1)n$ for some function $\omega(1)$ that tends to infinity with $n$. Taking log of both sides and rearranging gives the result. 
\end{proof}

\section{Conclusion}\label{sec:conclusion}

{ In this paper, we used tools from coding theory and number theory to construct Ruzsa-Szemer\'{e}di hypergraphs of girth 5.  These families of girth 5 hypergraphs prove Theorems \ref{thm: 456} and \ref{thm: graphs from codes}.  Now in \cite{TV} it was claimed that one can use Ruzsa's method for genus 2 equations to prove a $c_rN^{3/2 - o (1) }$ lower bound on $\textrm{ex}_r(N , \mathcal{C}_{g  < 5} )$ for all $r \geq 4$.  In an attempt to follow through on this claim, the problematic ``unbounded coefficient" equations 
(\ref{eq:bad 1}) and (\ref{eq:bad 2}) appear and this causes a significant gap in the purported method of proof.  We showed, however, that if one can find a Sidon set $\{\lambda_1,\cdots, \lambda_r\}$ with the square products property and that there is a dense enough sequence of primes certifying the property, then the claimed $N^{3/2-o(1)}$ bound follows. Furthermore, we constructed such sets when $r\in \{4,5,6\}$. We do not know if such sets do or do not exist for any other $r$.  It is also possible that a different modification of Rusza's method exists, and if so, it could have applications to coding theory.  We conclude this paragraph by commenting that our introduction contains several references which cite or use the claimed lower bound from \cite{TV}.}  



We also note that it is possible to extend Lemma \ref{lemma:multiple Behrend} to allow the coefficients to grow as long as they do not grow too fast, see for example Lemmas 3.4 and 3.5 in \cite{Ge} where equations are avoided over $[n]$ and coefficients are allowed to be $n^{o(1)}$. It is possible that allowing one to choose a growing sequence of $\vec{\lambda}$ which has some dependence on $N$ might be easier for some reason than choosing a fixed $\vec{\lambda}$ with the square products property. 

Alternatively, we also showed that one may push the Ruzsa method through for all $r$ if an equivalent of Lemma \ref{lemma:multiple Behrend} over $\mathbb{F}_p$ and without any restriction on the size of the coefficients could be found. We therefore state the following open problem; a positive resolution to it would give that $\mathrm{ex}_r(n, \mathcal{C}_{<5}) = N^{3/2-o(1)}$ for any $r$.

\begin{question}\label{behrend over field question}
Determine whether or not for any $\epsilon > 0$ there exists a $P$ such that for all primes $p \geq P$ and $a,b\in \mathbb{F}_p$, there exists a set $A\subset \mathbb{F}_p$ with $|A| > p^{1-\epsilon}$ and no nontrivial solutions to the equation $aX + bY = (a+b)Z$.
\end{question}

We believe that Question \ref{behrend over field question} is interesting in its own right. If $a$ and $b$ are of bounded ``height", then one can work in the first $\epsilon\cdot p$ elements of the field and pretend to work in the integers. 
A Behrend-type construction then gives a $p^{1-o(1)}$ lower bound. On the other side, a standard adaptation of the triangle removal lemma argument in the upper bound for the corners theorem gives an $o(p)$ upper bound for any $a,b$.

For all $r$ and $k$ one has $\mathrm{ex}_r(N, \mathcal{C}_{2k}) = O(N^{1+1/k})$ and it is a notorious open problem to determine whether there is a corresponding lower bound in almost all cases. However, when $r=2$ Conlon \cite{conlon2} showed that for any $k$ there exists a $c_k$ such that there exist $N$ vertex graphs with $\Omega(N^{1+1/k})$ edges and at most $c_k$ paths of length $k$ between any pair of vertices. The same result was shown for arbitrary $r$ in \cite{sunny}. Therefore, we may define the function $f(r,k)$ to be the minimum $c$ such that there exist $N$ vertex $r$-uniform hypergraphs with $\Omega(N^{1+1/k})$ edges and at most $c$ Berge paths of length $k$ between any pair of vertices. Understanding this function when $r=2$ was asked in \cite{conlon3}. The results in \cite{conlon2} and \cite{sunny} give that $f(r,k)$ is finite for any $r$ and $k$ but do not give any explicit bounds. In these papers, the random polynomial method was used and the best bound that could be extracted from them would give $f(r,k) = k^{O(k^2)}$ (the papers do not actually prove this bound). In \cite{ge2, ge3, ge1}, explicit bounds were given for $f(r,k)$ when $r\in \{2,3\}$ and $k=3$. 

In Section \ref{section: longer cycles}, we showed that one always sees trivial cycles of length $6$ in $\mathcal{H}(A,\vec{\lambda})$. However, if one could avoid {\em nontrivial} solutions to \eqref{a equation} when $k=6$, then the graph would have only a constant (depending only on $r$) number of paths of length $3$ between any pair of vertices. If the set avoiding nontrivial solutions to \eqref{a equation} had size $\Omega(|R|^{1/3})$ this would then give explicit upper bounds on $f(r,3)$ and would give lower bounds on some Tur\'an numbers for Berge theta graphs. It would be interesting to find such a set, though it might be hard as evidenced by the difficulty of finding $k$-fold Sidon sets, a problem which was proposed in \cite{LV} but is still open. It is possible that a smart choice of $\vec{\lambda}$ could make it easier. For longer cycles, one not only sees them but also sees unbounded numbers of paths between pairs of vertices: for $\ell \geq 4$ there will be pairs of vertices with $\Omega(|A|)$ paths of length $\ell$ between them even if $A$ avoids nontrivial solutions to \eqref{a equation}.



 \section*{Acknowledgements}
The second and third authors would like to thank Boris Bukh and Jacques Verstra\"ete for several helpful discussions on this problem over the last decade.


\bibliographystyle{plain}
	\bibliography{references.bib}
 
\section{Appendix} 

In order to confirm that the Sidon set $\{1,35,161,170,251,545 \}$ has the square products property, one must first consider all possible products of the form 
\[
( \lambda_{i_2} - \lambda_{i_1} ) ( \lambda_{i_3}  - \lambda_{i_2} ) ( \lambda_{i_4} - \lambda_{i_3} ) ( \lambda_{i_1} - \lambda_{i_4} )
\]
where $i_1,i_2,i_3,i_4$ are distinct indices in $\{1,2,3,4,5,6\}$ and $i_1 \neq i_2$, $i_2 \neq i_3$, $i_3 \neq i_4$, $i_4 \neq i_1$.  There are $\binom{6}{4}$ = 15 choices for the indices and then three ways to order the indices.  Indeed, if $1 \leq j_1 < j_2 < j_3 < j_4 \leq 6$ are the chosen indices, then the product 
\[
( \lambda_{j_2} - \lambda_{j_1} ) ( \lambda_{j_3}  - \lambda_{j_2} ) ( \lambda_{j_4} - \lambda_{j_3} ) ( \lambda_{j_1} - \lambda_{j_4} )
\]
will be negative.  The other two orderings lead to the positive products
\[
( \lambda_{j_2} - \lambda_{j_1} ) ( \lambda_{j_4}  - \lambda_{j_2} ) ( \lambda_{j_3} - \lambda_{j_4} ) ( \lambda_{j_1} - \lambda_{j_3} )
\]
and
\[
( \lambda_{j_3} - \lambda_{j_1} ) ( \lambda_{j_2}  - \lambda_{j_3} ) ( \lambda_{j_4} - \lambda_{j_2} ) ( \lambda_{j_1} - \lambda_{j_4} )
\]
Each of these positive products must be factored and then the non-square part must be put into the set $\mathcal{S}:=\mathcal{S}( ( 1,35,161,170,251,545 ) )$.  Each row in the table below lists the chosen indices, the two positive products coming from the different orderings, and finally the element(s) that must be placed in $\mathcal{S}$.   
\begin{footnotesize}
\begin{center}
\begin{tabular}{ |c  | c | c | c  | c|} \hline
$\{1,2,3,4 \}$ & $6609600 = 2^6\cdot 3^5\cdot 5^2\cdot 17^1$ & 
$459950400 =2^6\cdot 3^5\cdot 5^2\cdot 7^1\cdot 13^2$ & $3 \cdot 17$ and $3 \cdot 7$ \\ \hline
$\{1,2,3,5 \}$ & $105753600 = 2^{10}\cdot 3^5\cdot 5^2\cdot 17^1$ & 
$1088640000 = 2^{10}\cdot 3^5\cdot 5^4\cdot 7^1 $ & $3 \cdot 17$ and $3 \cdot 7$ \\ \hline
$\{1,2,3,6 \}$ & $1065369600 = 2^{14}\cdot 3^2\cdot 5^2\cdot 17^2$ & 
$5593190400 = 2^{12}\cdot 3^3\cdot 5^2\cdot 7^1\cdot 17^2$ & $3 \cdot 17$  \\ \hline
$\{1,2,4,5 \}$ & $100532016 = 2^4\cdot 3^7\cdot 13^2\cdot 17^1$ & 
$1232010000 = 2^4\cdot 3^6\cdot 5^4\cdot 13^2$ & $3 \cdot 17$  \\ \hline
$\{1,2,4,6 \}$ & $1098922500 = 2^2\cdot 3^2\cdot 5^4\cdot 13^2\cdot 17^2$ & 
$6329793600 = 2^6\cdot 3^4\cdot 5^2\cdot 13^2\cdot 17^2$ & None  \\ \hline
$\{1,2,5,6 \}$ & $1274490000 = 2^4\cdot 3^2\cdot 5^4\cdot 7^2\cdot 17^2$ & 
$14981760000 = 2^{10}\cdot 3^4\cdot 5^4\cdot 17^2$ & None  \\ \hline
$\{1,3,4,5 \}$ & $197121600 = 2^6\cdot 3^6\cdot 5^2\cdot 13^2$ & 
$34222500 = 2^2\cdot 3^4\cdot 5^4\cdot 13^2$ & None  \\ \hline
$\{1,3,4,6 \}$ & $3893760000 = 2^{12}\cdot 3^2\cdot 5^4\cdot 13^2$ & 
$317730816 = 2^{12}\cdot 3^3\cdot 13^2\cdot 17^1$ & $3 \cdot 17$  \\ \hline
$\{1,3,5,6 \}$ & $4515840000 = 2^{14}\cdot 3^2\cdot 5^4\cdot 7^2$ & 
$4700160000 = 2^{14}\cdot 3^3\cdot 5^4\cdot 17^1$ & $3 \cdot 17$  \\ \hline
$\{1,4,5,6 \}$ & $4658062500 = 2^2\cdot 3^2\cdot 5^6\cdot 7^2\cdot 13^2$ & 
$4131000000 = 2^6\cdot 3^5\cdot 5^6\cdot 17^1$ & $3 \cdot 17$  \\ \hline
$\{2,3,4,5 \}$ & $124002900 = 2^2\cdot 3^{11}\cdot 5^2\cdot 7^1$ & 
$23619600 = 2^4\cdot 3^{10}\cdot 5^2$ & $3 \cdot 7$  \\ \hline
$\{2,3,4,6 \}$ & $2449440000 = 2^8\cdot 3^7\cdot 5^4\cdot 7^1$ & 
$237945600 = 2^8\cdot 3^7\cdot 5^2\cdot 17^1$ & $3 \cdot 7$ and $3 \cdot 17$  \\ \hline
$\{2,3,5,6 \}$ & $3072577536 = 2^{12}\cdot 3^7\cdot 7^3$ & 
$3807129600 =2^{12}\cdot 3^7\cdot 5^2\cdot 17^1$ & $3 \cdot 7$ and $3 \cdot 17$  \\ \hline
$\{2,4,5,6 \}$ & $3214890000 = 2^4\cdot 3^8\cdot 5^4\cdot 7^2$ & 
$3346110000 = 2^4\cdot 3^9\cdot 5^4\cdot 17^1$ & $3 \cdot 17$  \\ \hline
$\{3,4,5,6 \}$ & $89302500 =2^2\cdot 3^6\cdot 5^4\cdot 7^2$ & 
$1049760000 = 2^8\cdot 3^8\cdot 5^4$ & None  \\ \hline
\end{tabular}
\end{center}
\end{footnotesize}

\end{document}